\documentclass[letterpaper, 10 pt, conference]{ieeeconf}  % Comment this line out
\IEEEoverridecommandlockouts                              % This command is only
\usepackage{graphicx}
\usepackage[]{mdframed}
\usepackage[colorlinks=true, allcolors=blue]{hyperref}
\usepackage{acronym}
\usepackage{times}
\usepackage{amsmath,dsfont}
\usepackage{amssymb,amsthm}
\usepackage{epsfig,verbatim}

\usepackage{setspace}
\usepackage{color}
\usepackage{epstopdf}
\usepackage{graphics}
\usepackage{accents}
\usepackage{url}
\usepackage{slashed}
\usepackage{stmaryrd}
\usepackage{color,soul} % assumes amsmath package installed
\usepackage{mathrsfs}
\usepackage{tikz}
\usepackage{mathtools}
\usepackage{scalerel}
\usepackage{bbm}
\usepackage{caption}
\usepackage{subcaption}
\usepackage{array}
\usepackage{pgfplots}
\pgfplotsset{compat=newest}
\usepackage{cite}
\usepackage{cuted}

\DeclareMathOperator*{\argmin}{arg\,min}

\newcommand{\brackets}[1]{\left(#1\right)}
\newcommand{\sbrackets}[1]{\left[#1\right]}
\newtheorem{theorem}{Theorem}[section]

\newtheorem{lemma}[theorem]{Lemma}

\acrodef{mmse}[MMSE]{{minimum mean-squared error}}
\acrodef{dm}[DM]{{decision maker}}
\acrodef{lqg}[LQG]{Linear-Quadratic-Gaussian}
\acrodef{pdf}[p.d.f.]{probability density function}
\acrodef{dpc}[DPC]{Dirty Paper Coding}

\title{\LARGE \bf
Low-Power Optimal Strategy for Witsenhausen Counterexample
}

\author{Mengyuan Zhao$^{1}$, Maël Le Treust$^{2}$, Tobias J. Oechtering$^{1}$% <-this % stops a space
\thanks{This work is supported by Swedish Research Council (VR) under grant 2020-03884. The work of M. Le Treust is supported in part by PEPR NF FOUNDS ANR-22-PEFT-0010.}% <-this % stops a space
\thanks{$^{1}$M. Zhao and T. J. Oechtering are with the Division of Information Science and Engineering,
        KTH Royal Institute of Technology, 100 44 Stockholm, Sweden
        {\tt\small \{mzhao, oech\}@kth.se }}
\thanks{$^{2}$M. Le Treust is with CNRS, Inria, IRISA UMR 6074, University of Rennes,
        F-35000 Rennes, France
        {\tt\small mael.le-treust@cnrs.fr}}
}
\allowdisplaybreaks 

\begin{document}

\maketitle
\thispagestyle{empty}
\pagestyle{empty}

%%%%%%%%%%%%%%%%%%%%%%%%%%%%%%%%%%%%%%%%%%%%%%%%%%%%%%%%%%%%%%%%%%%%%%%%%%%%%%%%
\begin{abstract}
We discuss the Witsenhausen counterexample from the perspective of varying power budgets and propose a low-power estimation (LoPE) strategy. Specifically, our LoPE approach designs the first decision-maker (DM) a quantization step function of the Gaussian source state, making the target system state a piecewise linear function of the source with slope one. This approach contrasts with Witsenhausen’s original two-point strategy, which instead designs the system state itself to be a binary step. While the two-point strategy can outperform the linear strategy in estimation cost, it, along with its multi-step extensions, typically requires a substantial power budget. Analogous to 
Binary Phase Shift Keying (BPSK) communication for Gaussian channels, we show that the binary LoPE strategy attains first-order optimality in the low-power regime, matching the performance of the linear strategy as the power budget increases from zero.
Our analysis also provides an interpretation of the previously observed near-optimal sloped step function ("sawtooth") structure to the Witsenhausen counterexample: In the low-power regime, power saving is prioritized, in which case the LoPE strategy dominates, making the system state a piecewise linear function with slope close to one. Conversely, in the high-power regime, setting the system state as a step function with the slope approaching zero facilitates accurate estimation. Hence, the sawtooth solution can be seen as a combination of both strategies.

\end{abstract}

%%%%%%%%%%%%%%%%%%%%%%%%%%%%%%%%%%%%%%%%%%%%%%%%%%%%%%%%%%%%%%%%%%%%%%%%%%%%%%%%
\section{Introduction}
%%%%%%%%%%%%%%%%%%%%%%%%%%%%%%%%%%%%%%%%%%%%%%%%%%%%%%%%%%%%%%%%%%%%%%%%%%%%%%%%%%%%%% General Intro & Trending approaches %%%%%%%%%%%%%%%%%%%%%%%%%%%%%%%%%%%%%%%%%%%%%%%%%%%%%%%%%%%%%%%%%%%%%%%%%%%%%%%%%%%%%%%%%%

Witsenhausen counterexample \cite{witsenhausen1968} has remained a fundamental challenge in decentralized control \cite{ho1972team,bansal1986stochastic, silva2010control, yuksel2013stochastic,gupta2015existence} for more than 55 years ever since it was introduced. Over the past few decades, considerable amount of approaches have been investigated through directions such as information theory based on vector-valued extension \cite{Grover2010Witsenhausen,agrawal2015implicit, Treust2024power,zhao2024coordination, grover2013approximately,zhao2025zero,zhao2024causal,zhao2024CDC}, approximation via numerical approach and learning \cite{romvary2015numerical, karlsson2011iterative,baglietto2001numerical,telsang2021numerical, mceneaney2015optimization}, signaling mechanism \cite{lee2001hierarchical,sahai1999information, wu2011transport}, and game theory \cite{ajorlou2020local,li2009learning}, etc. As illustrated in Fig.~\ref{fig: wits model}, the original Witsenhausen counterexample is formulated as a two-stage decentralized stochastic control problem with the following objective function:
\begin{align}
    \min_{\gamma_1,\gamma_2}\sbrackets{(k^2\cdot P + S)},\label{eq: obj 1}
\end{align}
where \vspace{-0.2cm}
\begin{align}
&\text{power cost: }&\quad &P= \mathbb E[U_1^2],\label{eq: power cost}\\
&\text{estimation cost: }&\quad&S = \mathbb E[(X_1-U_2)^2],\label{eq: est cost}
\end{align}
\vspace{-0.2cm}with\vspace{-0.2cm}
\begin{align}
 U_1 &= \gamma_1(X_0),\nonumber\\
 X_1 &= U_1+X_0\label{eq: X_1 generation},\\
 Y_1 &= X_1+Z_1,\nonumber \\
 U_2 &= \gamma_2(Y_1),\nonumber
\end{align}
and $X_0\sim\mathcal{N}(0,Q),
Z_1\sim\mathcal{N}(0,N)$, and $k^2$ is a trade-off constant that determines the relative importance of the action costs \eqref{eq: power cost}-\eqref{eq: est cost} across the two stages.

Here, we have two DMs, $\gamma_1$ and $\gamma_2$. The first DM $\gamma_1$ has perfect knowledge of the source state $X_0$ but acts with a power cost. The second DM $\gamma_2$ operates without a power cost but only observes the noisy version $Y_1 = X_1 + Z_1$. This feature that the subsequent DMs lack full access to the information held by prior DMs, is known as the \textit{nonclassical information pattern} \cite{mahajan2012information,ho1972team}. Despite its simplicity, this information structure turns the problem into non-convex, making the search for the universal optimality hard \cite{bansal1986stochastic, baglietto1997nonlinear}. 

Furthermore, for a given $\gamma_1$, the optimal $\gamma_2$ is given by the minimum mean-square error (MMSE) estimator of $X_1$ given $Y_1$ \cite{kay1993fundamentals}. As a result, the estimation cost $S$ becomes
\begin{align*}
    S = \mathbb E[(X_1-\mathbb E[X_1|Y_1])^2].
\end{align*}
Thus, the problem \eqref{eq: obj 1} effectively reduces to optimizing over only the strategy $\gamma_1$. This reduction imparts the dual role of control to the first DM: it needs to control the system state while transmitting relevant information simultaneously.

\begin{figure}[t]
      \centering
      % \documentclass{article}
% \usepackage{tikz}
% \usepackage{amssymb}

% \begin{document}

% \begin{figure}[!ht]
% \centering
\begin{tikzpicture}[scale=0.86]
    % Frames
    \draw (2,0) rectangle (3,1);
    \draw (6.8,0) rectangle (7.8,1);

    % Circles with Plus
    \draw (4.2,0.5) circle (0.2) node {$+$};
    \draw (5.6,0.5) circle (0.2) node {$+$};

    % Dots and Annotations
    \filldraw (1,1.5) circle (2pt) node[above] {$X_{0}\sim \mathcal{N}(0,Q)$};
    \filldraw (5.6,1.5) circle (2pt) node[above] {$Z_{1}\sim \mathcal{N}(0,N)$};

    % Arrows and Text
    \draw[->] (1,1.5) -- (1,0.5) -- (2,0.5);
    \draw[->] (1,1.5) -- (4.2,1.5) -- (4.2,0.7);
    \draw[->] (3,0.5) -- (4,0.5);
    \draw[->] (4.4,0.5) -- (5.4,0.5);
    \draw[->] (4.9,0.5) -- (4.9,-0.5) -- (8.8,-0.5);
    \draw[->] (5.6,1.5) -- (5.6,0.7);
    \draw[->] (5.8,0.5) -- (6.8,0.5);
    \draw[->] (7.8,0.5) -- (8.8,0.5);

    % More Annotations
    \node at (1.5,0.8) {$X_0$};
    \node at (3.5,0.8) {$U_{1}$};
    \node at (4.9,0.8) {$X_{1}$};
    \node at (6.3,0.8) {$Y_{1}$};
    \node at (8.3,0.8) {$U_{2}$};
    \node at (8.3,-0.2) {$X_{1}$};
    \node at (2.5,0.5) {$\gamma_1$};
    \node at (7.3,0.5) {$\gamma_2$};
\end{tikzpicture}
% \caption{The i.i.d. state and the channel noise are drawn according to Gaussian distributions $X_0^{n}\sim \mathcal{N}(0,Q\mathbb{I})$ and $Z_1^{n}\sim \mathcal{N}(0,N\mathbb{I})$.}
% \label{fig:sim}
% \end{figure}

% \end{document}
      \caption{\small{Witsenhausen Counterexample as a two-stage stochastic control problem with two DMs $\gamma_1,\gamma_2$.}}
      \label{fig: wits model}
      \vspace{-0.4cm}
   \end{figure}
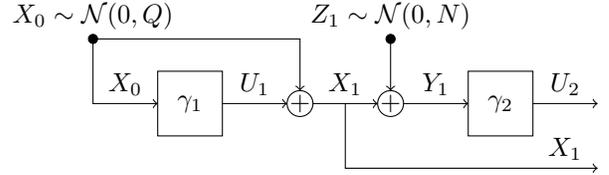

Witsenhausen \cite{witsenhausen1968} showed that the nonlinear two-point strategy $U_1 = a\cdot \text{sign}(X_0) - X_0$ for some $a\geq 0$ outperforms the best linear strategy when the trade-off constant $k^2$ is small, emphasizing estimation accuracy over power conservation. This nonlinear approach transforms the system state $X_1 = X_0+U_1$ to be binary, simplifying the estimation for the second DM \cite{grover2011source}. Generalizing this binary signaling idea, Lee et al. investigated a hierarchical search framework \cite{lee2001hierarchical} by taking $X_1$ to be an $n$-step function of $X_0$. While this approach improves estimation accuracy, it demands a big minimum-required power level. This minimum power level decreases as the number of quantization steps $n$ increases.

In this paper, instead of fixing the trade-off $k^2$, we discuss the original objective function \eqref{eq: obj 1} from a dynamic perspective and analyze how the estimation cost evolves as the available power varies, exploring this trend in the two-dimensional power-estimation cost region. More precisely, we look at the power-estimation trade-off in two regimes: 1) \textbf{power-sufficient regime} (i.e., $k^2$ small in \eqref{eq: obj 1}), where estimation accuracy is prioritized at the expense of enough power consumption. Most previous quantization-based approaches, such as \cite{karlsson2011iterative, lee2001hierarchical,vu2020hierachical}, focused on this regime. 2) \textbf{power-deficient regime} (i.e., $k^2$ big), where optimizing communication power usage becomes essential. In the second low-power regime, the optimal strategy acknowledged so far has been the optimal linear strategy, which was shown in \cite{wu2011transport} to be asymptotically optimal when the signal-to-noise ratio (SNR,  $\triangleq Q/N$) tends to zero. Later in this paper, we will further show that its estimation cost achieves the first-order optimal decay rate when power budget $P$ increases from zero, i.e., the first-order optimality.

As a main contribution of our work, we propose a novel low-power estimation (LoPE) strategy: Instead of $X_1$, we take the first DM's output $U_1$ a step function of $X_0$ with negative amplitude\footnote{This construction resembles that of \cite{lipsa2011optimal}, which challenges the optimality of linear schemes in a sequential LQG control problem}. As a result, the system state $X_1$ becomes a piecewise monotonically increasing linear function with slope one. Remarkably, even in its simplest form, so-called the Binary Phase Shift Keying (BPSK), LoPE attains first-order optimality in the low-power regime, matching the optimal linear scheme’s performance. Simulations further show that in certain cases, LoPE outperforms the optimal linear strategy in the low-power regime. Together with the hierarchical search method, we establish that linear controllers are strictly suboptimal in this case, providing an answer to the open question posed in \cite{wu2011transport} regarding the suboptimality of linear controllers for all $k^2>0$ at our chosen SNR.

In many approaches \cite{karlsson2011iterative,baglietto2001numerical,telsang2021numerical,lee2001hierarchical,sahai1999information,mceneaney2015optimization}, to name a few, a sloped step function (also refers to as a 'sawtooth '-type function) structure of $X_1$ was discovered to be a near-optimal solution. Our LoPE strategy provides new insights on the role of the slope in the sawtooth structure. Toward this end, we conclude that LoPE ($X_1$ a piecewise linear function with slope one) is communication-efficient and hence dominates the low-power regime, while hierarchical search ($X_1$ a perfect step function with slope zero) facilitates a more accurate estimation in the high-power regime. Hence, the near-optimal sawtooth solution could be viewed as a combination of both strategies.

% : Instead of transmitting $X_1$ to be a perfect step function of $X_0$, each staircase has a slope with the slope varying in different scenarios - bigger $k^2$, slope closer to 1; smaller $k^2$, slope closer to 0.
% This power efficiency trait for designing $X_1$ to be a piecewise linear function with slope $1$ in LoPE scheme opens the gate for interpreting the slope in the sloped step function result: In the low-power regime, a piecewise linear strategy with a slope close to 1 is power-efficient, and in the high-power scenario, reducing the slope to approach an ideal step function with zero slope facilitates more precise estimation by the second DM.

This paper is organized as follows. In Section~\ref{sec: low power est framework}, we introduce the LoPE strategy and derive closed-form expressions for its cost. Section~\ref{sec: first-order optimality} analyzes the first-order optimality of both the LoPE and the optimal linear policy. In Section~\ref{sec: empirical study}, we present numerical results and discuss some insights for the slope behavior observed in near-optimal solutions.

Throughout this paper, capital letters, e.g. $X_{0}$ denote random variables while lowercase letters, e.g. $x_0$ denote realizations. $\mathbb P(\cdot), f(\cdot)$ and $f(\cdot|\cdot)$ denote the general probability, probability density functions (PDFs) and conditional PDFs respectively. $\mathbf{1}_{\{\cdot\}}$ is the indicator function. $\Phi(\cdot)$, $\phi(\cdot)$ are the Gaussian cumulative density function (CDF) and PDF.

\section{Low-Power Estimation Framework}\label{sec: low power est framework}

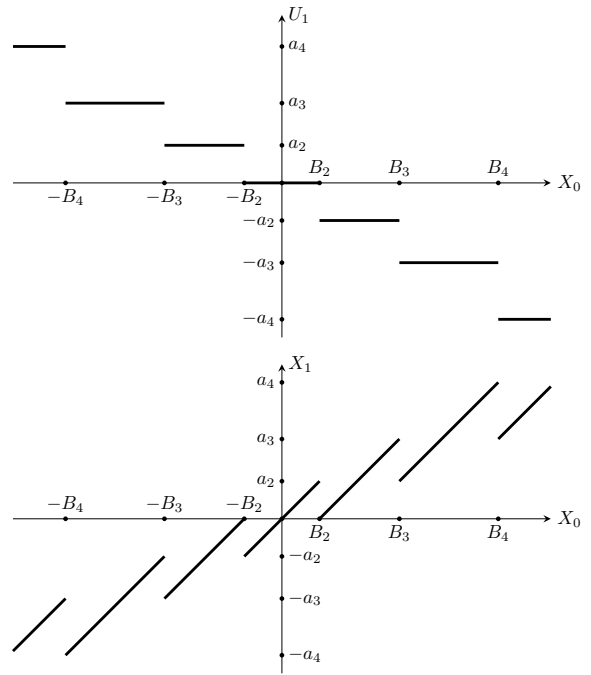
\begin{figure}[t]
  \centering
       \resizebox{0.9\linewidth}{!}{% \documentclass{article}
% \usepackage{tikz}

% \begin{document}

\begin{tikzpicture}[>=stealth, scale=1, every node/.style={scale=0.8}]

    % Draw the second x-y coordinate system (X_0, X_1) on the right
    \draw[->, thin] (-4,0) -- (4,0) node[right] {\(X_0\)};
    \draw[->, thin] (0,-2.3) -- (0,2.5) node[right] {\(U_1\)};

    % Draw X_1 function
    \draw[very thick] (0,0) -- (-0.5598,0);  % X_1 = -a, X_0 <= 0
    \draw[very thick]  (-0.5598, 0.5598) -- (-1.7472,0.5598);
    \draw[very thick] (-1.7472, 1.1874) -- (-3.2182, 1.1874);% X_1 = a, X_0 >= 0
    \draw[very thick]  (-3.2182,2.0307) -- (-4,2.0307);
    % \draw[red, dotted]  (-4, 4) -- (-5, 4);  % X_1 = a, X \draw[red, thick]_0 >= 0
    % \draw[red, thick] (-5, 5) -- (-6.5, 5);
    \draw[very thick] (0, 0) -- (0.5598, 0);
    \draw[very thick] (0.5598,-0.5598) -- (1.7472,-0.5598);
    \draw[very thick] (1.7472, -1.1874) -- (3.2182, -1.1874);
    \draw[very thick] (3.2182, -2.0307) -- (4, -2.0307);

    % Marking points at (0, a) and (0, -a)
    \fill (0,0) circle (1pt);
    \fill (0,0.5598) circle (1pt) node[right] {\(a_2\)};
    \fill (0,1.1874) circle (1pt) node[right] {\(a_3\)};
    \fill (0,2.0307) circle (1pt) node[right] {\(a_4\)};
    \fill (0,-0.5598) circle (1pt) node[left] {\(-a_2\)};
    \fill (0,-1.1874) circle (1pt) node[left] {\(-a_3\)};
    \fill (0,-2.0307) circle (1pt) node[left] {\(-a_4\)};
     \fill (0.5598,0) circle (1pt) node[above] {\(B_2\)};
    \fill (1.7472,0) circle (1pt) node[above] {\(B_3\)};
    \fill (3.2182,0) circle (1pt) node[above] {\(B_4\)};
    \fill (-0.5598,0) circle (1pt) node[below] {\(-B_2\)};
    \fill (-1.7472,0) circle (1pt) node[below] {\(-B_3\)};
    \fill (-3.2182,0) circle (1pt) node[below] {\(-B_4\)};
    % % \fill (-5,0) circle (2pt) node[above] {\(-B_5\)};
    % \fill (-5,0) circle (2pt) node[below] {\(-B_n\)};
    % \fill (1,0) circle (2pt) node[above] {\(B_1\)};
    % \fill (2,0) circle (2pt) node[above] {\(B_2\)};
    % \fill (3,0) circle (2pt) node[above] {\(B_3\)};
    % % \fill (4,0) circle (2pt) node[below] {\(B_4\)};
    % % \fill (5,0) circle (2pt) node[below] {\(B_5\)};
    % \fill (5,0) circle (2pt) node[above] {\(B_n\)};

%%%%%%%%%%%%%%%%%%%%%%%%%%%%%%%%%%%%%%%%%%%%%%%%%%%%%%%
     \draw[->, thin] (-4.0,-5.0) -- (4.0,-5.0) node[right] {\(X_0\)};
    \draw[->, thin] (0,-7.3) -- (0,-2.7) node[right] {\(X_1\)};

\draw[very thick] (0,-5) -- (-0.5598,-5-0.5598);  % X_1 = -a, X_0 <= 0
\draw[very thick] (-0.5598,-5) -- (-1.7472,0.5598-1.7472-5);

\draw[very thick] (-1.7472, 1.1874-1.7472-5) -- (-3.2182, 1.1874-3.2182-5);

    \draw[very thick]  (-3.2182,2.0307-3.2182-5) -- (-4,2.0307-9);
    % \draw[red, dotted]  (-4, 4) -- (-5, 4);  % X_1 = a, X \draw[red, thick]_0 >= 0
    % \draw[red, thick] (-5, 5) -- (-6.5, 5);
    \draw[very thick] (0, -5) -- (0.5598, -5+0.5598);
    \draw[very thick] (0.5598,-5) -- (1.7472,-0.5598+1.7472-5);
    \draw[very thick] (1.7472, -1.1874+1.7472-5) -- (3.2182, -1.1874+3.2182-5);
    \draw[very thick] (3.2182, -2.0307+3.2182-5) -- (4, -2.0307-1);

\fill (0.5598,-5) circle (1pt) node[below] {\(B_2\)};
    \fill (1.7472,-5) circle (1pt) node[below] {\(B_3\)};
    \fill (3.2182,-5) circle (1pt) node[below] {\(B_4\)};
    \fill (-0.5598,-5) circle (1pt) node[above] {\(-B_2\)};
    \fill (-1.7472,-5) circle (1pt) node[above] {\(-B_3\)};
    \fill (-3.2182,-5) circle (1pt) node[above] {\(-B_4\)};
    \fill (0,-5) circle (1pt);
    \fill (0,0.5598-5) circle (1pt) node[left] {\(a_2\)};
    \fill (0,1.1874-5) circle (1pt) node[left] {\(a_3\)};
    \fill (0,2.0307-5) circle (1pt) node[left] {\(a_4\)};
    \fill (0,-5.5598) circle (1pt) node[right] {\(-a_2\)};
    \fill (0,-6.1874) circle (1pt) node[right] {\(-a_3\)};
    \fill (0,-7.0307) circle (1pt) node[right] {\(-a_4\)};

\end{tikzpicture}

% \end{document}
} % Include your TikZ figure
  \caption{\small{Illustration of 4-step LoPE Scheme of $U_1$ (up) and $X_1=U_1+X_0$ (down) parameterized by $0 = a_1\leq a_2\leq a_3\leq a_4,0=B_1\leq B_2\leq B_3\leq B_4$.}}
  \label{fig:npointeq_U_1}
  \vspace{-0.3cm}
\end{figure}

We introduce the LoPE scheme by designing the output of the first DM $\gamma_1$ to be an $n$-step function of $X_0$: For a fixed tuple of $2n$ parameters $(\mathbf{a}_1^n = [a_1,...,a_n]^\top, \mathbf{B}_1^n = [B_1,...,B_n]^\top)$ satisfying 
\begin{align}
    &0\leq a_1\leq a_2\leq ...\leq a_n<\infty,\label{eq: a_i}\\
    &0=B_1\leq B_2\leq ...\leq B_n<\infty,\label{eq: B_i}
\end{align}
the $n$-step function $\gamma_1: X_0\rightarrow U_1$ is defined as follows 
\begin{align}
    U_1 &= \left\{
          \begin{aligned}
              &   +a_i    &\quad& \text{if } X_0\in(-B_{i+1},-B_i], \\
              &  -a_i   &\quad& \text{if } X_0\in[B_i,B_{i+1}), \\
          \end{aligned}   
        \right.\nonumber\\
        &\quad\quad\quad\quad\quad \quad\quad\quad  \text{for }i=1,...,n,\nonumber\\
        &=\sum_{i=1}^{n}a_i(\mathbf{1}_{X_0\in(-B_{i+1}, -B_i]} - \mathbf{1}_{X_0\in[B_i, B_{i+1})}),\label{eq: U_1}
\end{align}
where, for notational consistency, we define 
$B_{n+1}=\infty$. Note that, the amplitudes of $U_1$ have the opposite sign of the source $X_0$, making it a monotonically decreasing function.

Given \eqref{eq: X_1 generation} and \eqref{eq: U_1}, the system state becomes
\begin{align*}
    X_1 = \sum_{i=1}^{n}(X_0+a_i)(\mathbf{1}_{X_0\in(-B_{i+1}, -B_i]} - \mathbf{1}_{X_0\in[B_i, B_{i+1})}),
\end{align*}
which is a piecewise linear function of $X_0$. As a result, $X_1$ is symmetric around zero, and is also strictly monotonically increasing with slope one in each segment. Since $U_1$ has negative amplitudes, the segments in $X_1$ overlap and are concentrated around zero. An example of a 4-step LoPE mapping is shown in Fig.~\ref{fig:npointeq_U_1}, depicting both the step function $U_1$ and the resulting system state $X_1$.

In order to distinguish the LoPE scheme with previous quantization-based approaches, we now compare the simplest binary form of LoPE — the BPSK modulation (i.e., $n=1$) — with Witsenhausen’s two-point strategy. If we assume both strategies are parameterized by $a\geq 0$, their expressions are given as follows:
\begin{align}
        \text{BPSK: }\quad &U_1 = -a\cdot \text{sign}(X_0),\label{eq: BPSK}\\
        \text{Two-point: }\quad &U_1 = a\cdot\text{sign}(X_0)-X_0.\label{eq: 2-point}
    \end{align}
As we can see from Fig.~\ref{fig:comp2_point}, in the LoPE scheme, $U_1$ is itself a step function (with negative amplitude), whereas in Witsenhausen’s two-point strategy, the resulting system state $X_1$ is a step function.

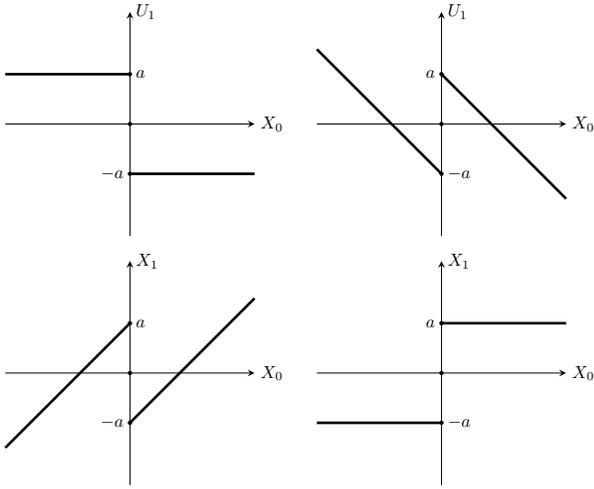
\begin{figure}[t]
  \centering
  \resizebox{0.93\linewidth}{!}{% \documentclass{article}
% \usepackage{tikz}

% \begin{document}

\begin{tikzpicture}[>=stealth, scale=1, every node/.style={scale=0.8}]

    % Draw the second x-y coordinate system (X_0, X_1) on the right
    \draw[->, thin] (-4,0) -- (0,0) node[right] {\(X_0\)};
    \draw[->, thin] (-2,-1.8) -- (-2,1.8) node[right] {\(U_1\)};

    % Draw X_1 function
    \draw[very thick] (-4,0.8) -- (-2,0.8);  % X_1 = -a, X_0 <= 0
    \draw[very thick]  (-2, -0.8) -- (-0,0-0.8);
    % \draw[very thick] (-1.7472, 1.1874) -- (-3.2182, 1.1874);% X_1 = a, X_0 >= 0
    % \draw[very thick]  (-3.2182,2.0307) -- (-4,2.0307);
    % % \draw[red, dotted]  (-4, 4) -- (-5, 4);  % X_1 = a, X \draw[red, thick]_0 >= 0
    % % \draw[red, thick] (-5, 5) -- (-6.5, 5);
    % \draw[very thick] (0, 0) -- (0.5598, 0);
    % \draw[very thick] (0.5598,-0.5598) -- (1.7472,-0.5598);
    % \draw[very thick] (1.7472, -1.1874) -- (3.2182, -1.1874);
    % \draw[very thick] (3.2182, -2.0307) -- (4, -2.0307);

    % Marking points at (0, a) and (0, -a)
    \fill (-2,0) circle (1pt);
    \fill (-2,0.8) circle (1pt) node[right] {\(a\)};
    \fill (-2,-0.8) circle (1pt) node[left] {\(-a\)};
    % \fill (0,2.0307) circle (1pt) node[right] {\(a_4\)};
    % \fill (0,-0.5598) circle (1pt) node[left] {\(-a_2\)};
    % \fill (0,-1.1874) circle (1pt) node[left] {\(-a_3\)};
    % \fill (0,-2.0307) circle (1pt) node[left] {\(-a_4\)};
    %  \fill (0.5598,0) circle (1pt) node[above] {\(B_2\)};
    % \fill (1.7472,0) circle (1pt) node[above] {\(B_3\)};
    % \fill (3.2182,0) circle (1pt) node[above] {\(B_4\)};
    % \fill (-0.5598,0) circle (1pt) node[below] {\(-B_2\)};
    % \fill (-1.7472,0) circle (1pt) node[below] {\(-B_3\)};
    % \fill (-3.2182,0) circle (1pt) node[below] {\(-B_4\)};
    % % \fill (-5,0) circle (2pt) node[above] {\(-B_5\)};
    % \fill (-5,0) circle (2pt) node[below] {\(-B_n\)};
    % \fill (1,0) circle (2pt) node[above] {\(B_1\)};
    % \fill (2,0) circle (2pt) node[above] {\(B_2\)};
    % \fill (3,0) circle (2pt) node[above] {\(B_3\)};
    % % \fill (4,0) circle (2pt) node[below] {\(B_4\)};
    % % \fill (5,0) circle (2pt) node[below] {\(B_5\)};
    % \fill (5,0) circle (2pt) node[above] {\(B_n\)};

%%%%%%%%%%%%%%%%%%%%%%%%%%%%%%%%%%%%%%%%%%%%%%%%%%%%%%%
     \draw[->, thin] (-4.0,-4.0) -- (0,-4.0) node[right] {\(X_0\)};
    \draw[->, thin] (-2,-4-1.8) -- (-2,-4+1.8) node[right] {\(X_1\)};

\draw[very thick] (-2,-3.2) -- (-4,-5.2);  % X_1 = -a, X_0 <= 0
\draw[very thick] (-2,-4.8) -- (0,-2.8);

% \draw[very thick] (-1.7472, 1.1874-1.7472-5) -- (-3.2182, 1.1874-3.2182-5);

%     \draw[very thick]  (-3.2182,2.0307-3.2182-5) -- (-4,2.0307-9);
%     % \draw[red, dotted]  (-4, 4) -- (-5, 4);  % X_1 = a, X \draw[red, thick]_0 >= 0
%     % \draw[red, thick] (-5, 5) -- (-6.5, 5);
%     \draw[very thick] (0, -5) -- (0.5598, -5+0.5598);
%     \draw[very thick] (0.5598,-5) -- (1.7472,-0.5598+1.7472-5);
%     \draw[very thick] (1.7472, -1.1874+1.7472-5) -- (3.2182, -1.1874+3.2182-5);
%     \draw[very thick] (3.2182, -2.0307+3.2182-5) -- (4, -2.0307-1);

% \fill (0.5598,-5) circle (1pt) node[below] {\(B_2\)};
%     \fill (1.7472,-5) circle (1pt) node[below] {\(B_3\)};
%     \fill (3.2182,-5) circle (1pt) node[below] {\(B_4\)};
%     \fill (-0.5598,-5) circle (1pt) node[above] {\(-B_2\)};
%     \fill (-1.7472,-5) circle (1pt) node[above] {\(-B_3\)};
%     \fill (-3.2182,-5) circle (1pt) node[above] {\(-B_4\)};
    \fill (-2,-4) circle (1pt);
    \fill (-2,-3.2) circle (1pt) node[right] {\(a\)};
    \fill (-2,-4.8) circle (1pt) node[left] {\(-a\)};
%     \fill (0,2.0307-5) circle (1pt) node[left] {\(a_4\)};
%     \fill (0,-5.5598) circle (1pt) node[right] {\(-a_2\)};
%     \fill (0,-6.1874) circle (1pt) node[right] {\(-a_3\)};
%     \fill (0,-7.0307) circle (1pt) node[right] {\(-a_4\)};

\draw[->, thin] (1,0) -- (5,0) node[right] {\(X_0\)};
    \draw[->, thin] (3,-1.8) -- (3,1.8) node[right] {\(U_1\)};
    \draw[very thick] (1,1.2) -- (3,-0.8);  % X_1 = -a, X_0 <= 0
    \draw[very thick] (3,0.8) -- (5,-1.2);  % X_1 = -a, X_0 <= 0

 \fill (3,0) circle (1pt);
    \fill (3,0.8) circle (1pt) node[left] {\(a\)};
    \fill (3,-0.8) circle (1pt) node[right] {\(-a\)};

 \draw[->, thin] (1,-4.0) -- (5,-4.0) node[right] {\(X_0\)};
    \draw[->, thin] (3,-5.8) -- (3,-2.2) node[right] {\(X_1\)};
\draw[very thick] (1,-4.8) -- (3,-4.8);  % X_1 = -a, X_0 <= 0
\draw[very thick] (3,-3.2) -- (5,-3.2);  % X_1 = -a, X_0 <= 0

     \fill (3,-4) circle (1pt);
    \fill (3,-3.2) circle (1pt) node[left] {\(a\)};
    \fill (3,-4.8) circle (1pt) node[right] {\(-a\)};

\end{tikzpicture}

% \end{document}
} % Include your TikZ figure
  \caption{\small{Illustration of the BPSK LoPE Scheme (left) defined in \eqref{eq: BPSK} and Witsenhausen's two-point scheme (right) defined in \eqref{eq: 2-point}. In LoPE, $U_1$ is defined as a step function, whereas in the two-point strategy, $X_1$ results in a step function instead.}}
  \label{fig:comp2_point}
  \vspace{-0.3cm}
\end{figure}

Furthermore, if we define the probability that $X_0$ lies in these quantization intervals by
\begin{align*}
    p_i& \triangleq \mathbb P(X_0\in(-B_{i+1}, -B_i])= \mathbb P(X_0\in[B_i, B_{i+1}))\\
       &=\Phi\brackets{\frac{B_{i+1}}{\sqrt{Q}}} - \Phi\brackets{\frac{B_{i}}{\sqrt{Q}}},\quad \text{for }i=1,...,n,
\end{align*}
the power cost induced by scheme \eqref{eq: U_1} is hence
\begin{align*}
    P = \mathbb E[U_1^2] = 2\sum_{i=1}^n a_i^2p_i.
\end{align*}

The probability distribution of $X_1$ takes the following piecewise-defined Gaussian form
\vspace{-0.15cm}
\begin{align}
    f_{X_1}(x) &= \frac{1}{\sqrt{Q}}\sum_{i=1}^n \left[\phi\brackets{\frac{x-a_i}{\sqrt{Q}}}\mathbf{1}_{x\in(-B_{i+1}+a_i, -B_i+a_i]}\right.\nonumber \\
    &\quad \left.+ \phi\brackets{\frac{x+a_i}{\sqrt{Q}}}\mathbf{1}_{x\in[B_i-a_i, B_{i+1}-a_i) }\right].\label{eq: distr X_1}
\end{align}
These pieces involved in \eqref{eq: distr X_1} overlap, leading to regions where multiple Gaussian components sum and contribute to the overall density. As we will see in Sec. \ref{sec: empirical study}, these optimized overlapping segments collectively reinforce the probability mass at zero, making the distribution of $X_1$ concentrate. 

The resulting power and estimation costs induced by LoPE scheme \eqref{eq: U_1} are summarized in Theorem \ref{thm: lope}. The derivation of the estimation cost is detailed in Appendix A.

\begin{theorem}\label{thm: lope}
For $n\in\mathbb N$ and a fixed tuple of  parameters $(\mathbf{a}_1^n = [a_1,...,a_n]^\top, \mathbf{B}_1^n = [B_1,...,B_n]^\top)$ satisfying \eqref{eq: a_i} and \eqref{eq: B_i} the LoPE scheme has the power and estimation costs of:\
\begin{align}
    P(\mathbf{a}_1^n, \mathbf{B}_1^n)& = 2\sum_{i=1}^n a_i^2p_i,\label{eq: power cost n step}\\
    S(\mathbf{a}_1^n, \mathbf{B}_1^n)&= Q - 4\sqrt{Q}\sum_{i=1}^na_i\brackets{\phi\brackets{\frac{B_i}{\sqrt{Q}}} - \phi\brackets{\frac{B_{i+1}}{\sqrt{Q}}}}\nonumber\\
    &\quad +2\sum_{i=1}^na_i^2p_i - \int \frac{\brackets{\sum_{i=1}^n (E_{-i}+E_i)}^2}{\sum_{i=1}^n (F_{-i}+F_i)} dy,\label{eq: est cost n step}
\end{align}
where quantities $E_{-i}, E_i, F_{-i}, F_i$ for $i=1,...,n$ are given in \eqref{eq: F_-i} - \eqref{eq: E_i}.
    
\end{theorem}

\begin{figure*}[t]
\begin{align}
    &F_{-i}\!=\!\sqrt{\frac{1}{Q+N}}\phi\brackets{\frac{y-a_i}{\sqrt{Q+N}}} \!\sbrackets{\Phi\!\brackets{\!\sqrt{\frac{Q+N}{QN}}\brackets{B_{i+1}+\frac{Q(y-a_i)}{Q+N}}}\!-\! \Phi\brackets{\sqrt{\frac{Q+N}{QN}}\brackets{B_{i}+\frac{Q(y-a_i)}{Q+N}}}}, \label{eq: F_-i}\\
    &F_i = \sqrt{\frac{1}{Q+N}}\phi\brackets{\frac{y+a_i}{\sqrt{Q+N}}}\!\sbrackets{\!\Phi\brackets{\sqrt{\frac{Q+N}{QN}}\brackets{B_{i+1}- \frac{Q(y+a_i )}{Q+N}}}\!-\! \Phi\brackets{\sqrt{\frac{Q+N}{QN}}\brackets{B_{i}- \frac{Q(y+a_i )}{Q+N}}}}, \label{eq: F_i} \\
    & E_{-i}=\frac{a_iN + yQ}{Q+N}F_{-i} - \frac{\sqrt{QN}}{2\pi(Q+N)} \left[ \exp\brackets{-\frac{Q(-B_i+a_i-y)^2 + NB_i^2}{2QN}  }   - \exp\brackets{-\frac{Q(-B_{i+1}+a_i-y)^2 + NB_{i+1}^2}{2QN}  }    \right],\label{eq: E_-i}\\
    &E_i=\frac{-a_iN + yQ}{Q+N}F_{i} - \frac{\sqrt{QN}}{2\pi(Q+N)}\left[ \exp\brackets{-\frac{Q(B_{i+1}-a_i-y)^2 + NB_{i+1}^2}{2QN}  } - \exp\brackets{-\frac{Q(B_{i}-a_i-y)^2 + NB_{i}^2}{2QN}  }    \right].\label{eq: E_i}
\end{align}
\vspace{-0.5em}
\noindent\rule{\textwidth}{0.4pt}
\end{figure*}

\section{First-order Optimality}\label{sec: first-order optimality}
It is clear that at $P=0$, the first DM $\gamma_1$ takes no action, i.e., $U_1=0$, and thus the system state becomes $X_1 = X_0$. In such case, the second DM $\gamma_2$ must estimate the intact continuous Gaussian source state $X_0$ based on its noisy observation. Hence, the estimation cost at $P=0$ is $S(P) = \frac{QN}{Q+N}$, which is precisely the MMSE of estimating Gaussian $X_0\sim\mathcal{N}(0,Q)$ from the output of the additive white Gaussian noise (AWGN) channel $Y_1 = X_0+Z_1$ with noise $Z_1\sim\mathcal{N}(0,N)$. As the available power $P$ begins to increase from zero, the first DM starts to work at the expense of power consumption, which in turn improves the estimation accuracy and decreases the estimation cost.

Our step function LoPE strategy corresponds to discrete signaling over a continuous channel. In the seminal work \cite{verdu2002spectral}, Verdú showed that the discrete BPSK signaling strategy over an AWGN channel achieves in the low-power regime the first-order optimality considering spectral efficiency, i.e., a simple binary signaling scheme is sufficiently communication-efficient in the low power regime. Even though the Witsenhausen setup corresponds to an AWGN channel with an additive state $X_0$, our result summarized in Theorem \ref{thm: first order opt} shows that in this distinct setting, the discrete LoPE scheme still achieves first-order optimality in estimation cost as power increases from zero, matching the performance of the optimal linear strategy. Note that, this first-order optimality is a fundamental necessary condition that the universally optimal control strategy must satisfy.

% In this section, we establish that both the LoPE strategy and the optimal linear strategy achieve the optimal first-order decay rate in estimation cost when the power increases from zero, see Theorem \ref{thm: first order opt}. 

\begin{theorem}\label{thm: first order opt}
Both the optimal linear strategy and the LoPE strategy satisfy the first-order optimality criterion as $P$ increases from zero, specifically, their estimation cost derivatives at $P=0$ satisfy $S'(P) = -\infty$
\end{theorem}

The proof of this theorem is given in Appendix B. 
This result implies that it is not necessary to communicate continuous linear inputs: Discrete, or even binary, input strategies can be equally effective in the low-power regime.

\newpage
\section{Parameter Optimization and Empirical Study}\label{sec: empirical study}
In our simulation, we slightly modify the objective function \eqref{eq: obj 1} to better explore and visualize the trade-off between the two-stage costs: Instead of fixing the weight for the second estimation cost to be $1$ and varying only the first trade-off constant $k^2$, we normalize both terms by dividing by $k^2+1$. This is equivalent to a weighted optimization problem parameterized by $\omega \triangleq \frac{k^2}{k^2+1}\in[0,1]$, each of which corresponds to a supporting hyperplane in a two-dimensional power-estimation cost space. Specifically, for fixed $k^2\geq 0$ and quantization level $n\in\mathbb N$, we determine the optimal parameters $(\mathbf{a}_1^n = [a_1,...,a_n]^\top, \mathbf{B}_1^n = [B_1,...,B_n]^\top)$ satisfying \eqref{eq: a_i} and \eqref{eq: B_i} such that
\begin{align*}
    (\mathbf{a}^*_{\omega}, \mathbf{B}^*_{\omega})& =  \argmin\frac{1}{k^2+1}\sbrackets{ k^2 P(\mathbf{a}_1^n,\mathbf{B}_1^n) + S(\mathbf{a}_1^n,\mathbf{B}_1^n) }\\
    &=
    \argmin\sbrackets{ \omega P(\mathbf{a}_1^n,\mathbf{B}_1^n) + (1-\omega) S(\mathbf{a}_1^n,\mathbf{B}_1^n) },
\end{align*}
which gives us the best configuration for the given parameter $k^2$. Then, we plot the curve $\brackets{P(\mathbf{a}^*_\omega, \mathbf{B}^*_\omega), S(\mathbf{a}^*_\omega, \mathbf{B}^*_\omega)}\in\mathbb R^2_+$ for all $\omega=\frac{k^2}{k^2+1}\in[0,1]$, which gives us the optimized estimation cost value at each specific power consumption.

We now examine how the parameters $(\mathbf{a}_1^n,\mathbf{B}_1^n)$ of the LoPE strategy influence the distribution of the system state $X_1$ in \eqref{eq: distr X_1} under different power constraints. Fixing $Q=1,N=0.1$, we compare two different LoPE schemes: 1) the simplest BPSK strategy is defined in \eqref{eq: BPSK}, and 2) a more refined 8-step LoPE strategy with more tuning parameters. 

First, consider the power-sufficient case when $P=0.93$.  Fig. \ref{fig: distr_X_1} shows the resulting distributions of $X_1$ for both strategies, with parameters optimized for the given conditions. As we can see, when power is sufficient, as the number of steps $n \in \mathbb{N}$ (hence the number of tuning parameters) increases, the shape of the distribution of $X_1$ becomes more concentrated around zero with larger amplitude. This concentration facilitates a more deterministic and accurate estimation for the second DM. Therefore, in the high-power scenario, the LoPE strategy enables highly effective estimation performance, especially with more quantization steps. Furthermore, with $n\rightarrow\infty$, the distribution of $X_1$ converges to a Dirac delta function centered at zero.

% \documentclass{article}
% \usepackage{pgfplots}
% \usepackage{tikz}
% \pgfplotsset{compat=newest}

% \begin{document}

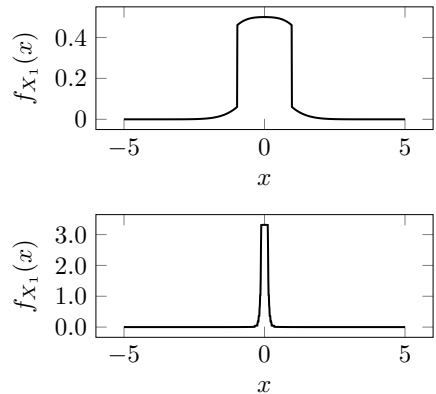
\begin{figure}
    \centering
    % First TikZ figure
\begin{tikzpicture}[scale=0.95]
        \begin{axis}[
            width=6.3cm, height=3.3cm,
            xlabel={$x$},
            ylabel={$f_{X_1}(x)$},
        ]

        % MAIN PLOT: load data from fX1_data.dat
        \addplot[thick] table [x=x, y=fX, col sep=space] {fig/fX1_data_BPSK.dat};
        % \addlegendentry{$f_{X_1}(x)$}

        \end{axis}
    \end{tikzpicture}

    \vspace{0.28cm} % Adds vertical space between the two plots

    % Second TikZ figure
    \begin{tikzpicture}[scale=0.95]
    \centering
        \begin{axis}[
            width=6.3cm, height=3.3cm,
            xlabel={$x$},
            ylabel={$f_{X_1}(x)$},
            yticklabel style={
    /pgf/number format/fixed,
    /pgf/number format/precision=1,
    /pgf/number format/fixed zerofill
}
        ]

        % MAIN PLOT: load data from fX1_data.dat
        \addplot[thick] table [x=x, y=fX, col sep=space] {fig/fX1_data_n=8.dat};
        % \addlegendentry{$f_{X_1}(x)$}

        \end{axis}
    \end{tikzpicture}

    \caption{\small{Optimized distribution of $f_{X_1}(x)$ for BPSK and 8-step LoPE strategies under $P=0.93$, $Q=1$, $N=0.1$.}}
    \label{fig: distr_X_1}
    \vspace{-0.3cm}
\end{figure}

% \documentclass{article}
% \usepackage{pgfplots}
% \usepackage{caption}
% \usepackage{subcaption}
% \pgfplotsset{compat=newest}

% \begin{document}

% \begin{figure}
%     \centering
    
%     % First TikZ plot
%     \begin{subfigure}[b]{\textwidth}
%         \centering
%         \begin{tikzpicture}
%             \begin{axis}[
%                 width=0.75\textwidth, height=5cm,
%                 xlabel={$x$},
%                 ylabel={$f_{X_1}(x)$},
%             ]
%             \addplot[thick] table [x=x, y=fX, col sep=space] {fig/Tikz/fX1_data_BPSK.dat};
%             \end{axis}
%         \end{tikzpicture}
%         \caption{BPSK strategy ($n=2, B_2=0$)}
%     \end{subfigure}
    
%     \vspace{0.5cm}  % adjust spacing between plots
    
%     % Second TikZ plot
%    \begin{subfigure}[b]{\textwidth}
%     \centering
%     \begin{tikzpicture}
%         \begin{axis}[
%             width=0.75\textwidth, height=5cm,
%             xlabel={$x$},
%             ylabel={$f_{X_1}(x)$},
%            yticklabel style={
%     /pgf/number format/fixed,
%     /pgf/number format/precision=1,
%     /pgf/number format/fixed zerofill
% }
%         ]
%         \addplot[thick] table [x=x, y=fX, col sep=space] {fig/Tikz/fX1_data_n=8.dat};
%         \end{axis}
%     \end{tikzpicture}
%     \caption{8-step LoPE strategy}
% \end{subfigure}

%     \caption{Optimized distribution of $f_{X_1}(x)$ for BPSK and 8-step LoPE strategies under $P=0.93$, $Q=1$, $N=0.1$.}
%     \label{fig: distr_X_1}
% \end{figure}

% \end{document}

In contrast, in the low-power regime, the power budget restricts the shaping flexibility for the parameters. Hence, the system state $X_1$ remains closer to the continuous source state $X_0$. As a result, reconstructing $X_1$ accurately is more challenging for $\gamma_2$. In such cases, LoPE becomes more power-saving and communication-efficient than the linear scheme, since the first DM $\gamma_1$ transmits a discrete signal \eqref{eq: U_1} rather than the (scaled) continuous Gaussian state $X_0$.

Fig. \ref{fig: performance comparison} compares the simulation performance in the low-power regime for $Q=1,N=0.1$, evaluating the cost functions for the 2-step and 4-step LoPE strategies (i.e., $n=2$ and $n=4$), denoted by $S_2(P),S_4(P)$, the optimal linear strategy  $S_{\ell}(P) $ in \cite[Lemma 11]{witsenhausen1968}, and the optimal Gaussian strategy $S_{\mathsf{G}}(P)$ in \cite[Lemma 12]{witsenhausen1968}, which serves as the convex envelope of the optimal linear strategy. As shown in the figure, they all exhibit an estimation cost decay rate of $-\infty$ as $P$ increases from zero, consistent with our analysis in Sec.\ref{sec: first-order optimality}. Moreover, when $P\leq 0.171$, the 4-step LoPE strategy outperforms the optimal Gaussian and linear schemes. 

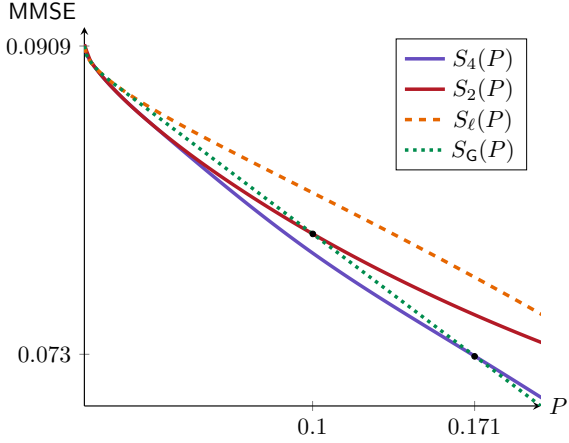
\begin{figure}[t]
      \centering
      \resizebox{0.47\textwidth}{!}{% \documentclass{article}
% \usepackage{pgfplots}
% \pgfplotsset{compat=1.17}
% \usepackage{xcolor}

\definecolor{airforceblue}{rgb}{0.9, 0.4, 0.0}      % More saturated, vivid blue
\definecolor{antiquebrass}{rgb}{0.0, 0.55, 0.3}         % Warmer, more contrast
\definecolor{alizarin}{rgb}{0.7, 0.1, 0.15}          % Sharper, more vivid red-orange
\definecolor{amethyst}{rgb}{0.4, 0.3, 0.75}       % Deeper purple, more academic

% \begin{document}

\begin{tikzpicture}[scale=0.95]
\begin{axis}[
    xlabel={$P$},
    ylabel={$\mathsf{MMSE}$},
    legend pos=north east,
    axis lines=middle,
    axis line style={black, line width=0.5pt},
    xmin=0, xmax=0.20,
    ymin=0.070, ymax=0.092,
    xtick={0.1, 0.171},
    ytick={0.073, 0.0909},
    xlabel style={at={(ticklabel* cs:1)}, anchor=west},
    ylabel style={at={(ticklabel* cs:1)}, anchor=south east},
    restrict x to domain=0:0.5,
    restrict y to domain=0.04:0.091,
    % grid=both,  % <--- This enables both major and minor grid lines
    % minor grid style={dashed, lightgray},  % Style of minor grid lines
    % major grid style={solid, gray},        % Style of major grid lines
   scaled x ticks=false,
xticklabel style={/pgf/number format/.cd, fixed, precision=4},
scaled y ticks=false,
yticklabel style={/pgf/number format/.cd, fixed, precision=4},
]

% Plot S_{\mathsf{coord2}}(P)
\addplot[amethyst, line width=1.5pt, solid, smooth, samples=100] table [col sep=comma, x index=0, y index=1] {fig/4eq.csv};
\addlegendentry{$S_4(P)$}

% Plot S_2(P)
\addplot[alizarin, line width=1.5pt, solid, smooth, samples=100] table [col sep=comma, x index=0, y index=1] {fig/2eq.csv};
\addlegendentry{$S_2(P)$}

% Plot S_{\ell}(P)
\addplot[airforceblue, line width=1.5pt, dashed] table [col sep=comma, x index=0, y index=1] {fig/linear.csv};
\addlegendentry{$S_{\ell}(P)$}

% % Plot S_{\mathsf{G}}(P)
\addplot[antiquebrass, line width=1.5pt, dotted] table [col sep=comma, x index=0, y index=1] {fig/Gaussian.csv};
\addlegendentry{$S_{\mathsf{G}}(P)$}

% Highlight key points
\addplot[only marks, mark=*, mark options={scale=0.6}, color=black] coordinates {
    (0.1001, 0.07999) (0.171, 0.07288)
};

\end{axis}
\end{tikzpicture}

% \end{document}
 }
      \caption{\small{Performance comparison of different control strategies in the low-power high-estimation cost regime when $Q=1,N=0.1$.}}
      \label{fig: performance comparison}
      \vspace{-0.3cm}
   \end{figure}

In \cite{vu2020hierachical}, the authors proved that optimal linear controllers are strictly suboptimal in the high SNR regime ($Q/N$ large) when $k^2 < 0.318$ and posed the open question of whether linear controllers are also strictly suboptimal for all ratio $Q/N $ and $k^2 > 0$. Our simulation results show that the trade-off parameters for which the 4-step LoPE strictly surpasses the optimal Gaussian and linear schemes correspond approximately to $k^2\in[0.21,+\infty)$ in \eqref{eq: obj 1}. Moreover, for $k^2 \in[0,0.39]$, the 2-step hierarchical search scheme can be shown to outperforms the best linear strategy. Our findings provide a concrete partial answer to this open question: when $Q/N=10$ for all $k^2 > 0$, the linear controllers are strictly suboptimal, and are outperformed by the combination of the LoPE strategy and the hierarchical search scheme.

In many prior attempts to study Witsenhausen counterexample, a recurring discovery is that near-optimal solutions resemble a slightly sloped step function, also referred to as a sawtooth structure. Several works \cite{lee2001hierarchical, karlsson2011iterative, telsang2021numerical} attempted to approximate such structure by subdividing the existing quantization steps into smaller staircase-like segments. Moreover, in \cite{baglietto1997nonlinear}, the dependence of the slope on the trade-off parameter $k^2$ in \eqref{eq: obj 1} has been first noticed:
When power conservation is prioritized (e.g. $k^2=10$ in \cite[Fig. 4]{baglietto1997nonlinear}), the resulting control function resembles a linear strategy that has a steeper slope, approximately equal to one. Conversely, when estimation accuracy is prioritized (e.g. $k^2=0.1$ in \cite[Fig. 5]{baglietto1997nonlinear}), the near optimal solution approaches an ideal step function with slope small, see also \cite[Fig. 4]{baglietto2001numerical}. Hence, a natural guess is: the more the controller must save power, the larger the slope the control strategy tends to adopt.

Our analysis and simulation results support this interpretation of the role of the slope: In low-power regimes, the most communication-efficient approach is for the first DM to transmit discrete signals, making $U_1$ a step function of $X_0$. This results in $X_1$ a piecewise linear function of $X_0$ with slope one, see Fig. \ref{fig:npointeq_U_1} and \ref{fig:comp2_point} (left). Conversely, when power is sufficient and estimation precision becomes more crucial, $\gamma_1$ can afford more power to cancel out the continuous source state. Outputting a discrete $X_1$ as an ideal step function of $X_0$ with slope zero (see also Fig. \ref{fig:comp2_point}, right) simplifies and improves the estimation for the second DM. These findings suggest that this near-optimal sawtooth solution can be viewed as a combination of these two approaches: LoPE strategy dominates the low-power regime, while designing the system state as an ideal step function is preferred in the high-power regime.

\section*{Appendix A: Derivation of Theorem \ref{thm: lope}}\label{app a}
In order to derive the quantities \eqref{eq: F_-i} - \eqref{eq: E_i}, and hence the estimation cost \eqref{eq: est cost n step}, we need the following lemma which gives the explicit expressions of the Gaussian integrations
\begin{lemma}
    For appropriately chosen parameters, the closed-form results of the Gaussian integrations are
    \begin{align}
        I_1 = &\int_{A}^B \exp{\brackets{-ax^2 + bx + c}}dx\nonumber\\
        &= \sqrt{\frac{\pi}{a}}\exp{\brackets{c +
        \frac{b^2}{4a}}}\left[\Phi\brackets{\sqrt{2a}\brackets{B-\frac{b}{2a}}} \right.\nonumber\\
        &\quad \left.- \Phi\brackets{\sqrt{2a}\brackets{A-\frac{b}{2a}}}\right],\label{eq: I_1}\\
        I_2 = &\int_A^B x\cdot \exp{\brackets{-ax^2 + bx + c}}dx\nonumber\\
        &= \frac{1}{2a}\cdot \left[ b\cdot I_1 - (\exp{(-aB^2 + bB + c)}\right.\nonumber\\
        &\quad \left.- \exp{(-aA^2 + bA + c)})\right].\label{eq: I_2}
    \end{align}
\end{lemma}

\begin{proof}[Derivation of the estimation cost in Theorem \ref{thm: lope}]
    
 Let's continue with the distribution of $X_1$ given in \eqref{eq: distr X_1}. Since $Y_1 = X_1+Z_1$ where $Z_1\sim\mathcal{N}(0,N)$, we have
\begin{align*}
&f_{Y_1|X_1}(y|x) = \frac{1}{\sqrt{N}}\phi\brackets{\frac{y-x}{\sqrt{N}}},\\
    &f_{X_1,Y_1}(x,y)\\
    &=\! \frac{1}{\sqrt{QN}}\!\phi\brackets{\frac{y-x}{\sqrt{N}}}\!\sum_{i=1}^n \left[\phi\brackets{\frac{x-a_i}{\sqrt{Q}}}\!\mathbf{1}_{x\in(-B_{i+1}+a_i, -B_i+a_i]}\right. \\
    &\quad \left.+ \phi\brackets{\frac{x+a_i}{\sqrt{Q}}}\mathbf{1}_{x\in[B_i-a_i, B_{i+1}-a_i) }\right].
\end{align*}
Marginalization gives us the probability distribution of $Y_1$
\begin{align}
    f_{Y_1}(y) &= \int_{-\infty}^\infty f_{X_1,Y_1}(x,y) dx\nonumber\\
    &= \sum_{i=1}^n\underbrace{\int_{-B_{i+1}+a_i}^{-B_i+a_i}\frac{1}{\sqrt{QN}}\phi\brackets{\frac{x-a_i}{\sqrt{Q}}}\phi\brackets{\frac{y-x}{\sqrt{N}}} dx}_{F_{-i}}\nonumber\\
    &\quad +\sum_{i=1}^n\underbrace{\int_{B_{i}-a_i}^{B_{i+1}-a_i}\frac{1}{\sqrt{QN}}\phi\brackets{\frac{x+a_i}{\sqrt{Q}}}\phi\brackets{\frac{y-x}{\sqrt{N}}} dx}_{F_{i}}\nonumber\\
    &=\sum_{i=1}^n (F_{-i} + F_i). \label{eq: f_Y_1}
\end{align}
Using the formula \eqref{eq: I_1}, we can get the closed-form expressions of the integral $F_{-i},F_i$ for $i=1,...,n$ given in \eqref{eq: F_-i}

The MMSE estimation of $X_1$ given a realization of $Y_1=y$ is therefore
\begin{align*}
    &\mathbb E[X_1|Y_1 = y] \\
    &= \int_{-\infty}^\infty x\cdot f_{X_1|Y_1}(x|y) dx\\
    &=\frac{1}{f_{Y_1}(y)}\int_{-\infty}^\infty xf_{X_1,Y_1}(x,y)dx\\
    &=\!  \frac{1}{f_{Y_1}(y)}\!\left[\!\sum_{i=1}^n\!\underbrace{\int_{-B_{i+1}+a_i}^{-B_i+a_i}\!\frac{x}{\sqrt{QN}}\phi\brackets{\frac{x-a_i}{\sqrt{Q}}}\!\phi\brackets{\frac{y-x}{\sqrt{N}}}\!dx}_{E_{-i}}\right.\\
    &\quad \left. + \sum_{i=1}^n\underbrace{\int_{B_{i}-a_i}^{B_{i+1}-a_i}\frac{x}{\sqrt{QN}}\phi\brackets{\frac{x+a_i}{\sqrt{Q}}}\phi\brackets{\frac{y-x}{\sqrt{N}}} dx}_{E_{i}}\right]\\
   &= \frac{\sum_{i=1}^n (E_{-i}+E_i)}{\sum_{i=1}^n (F_{-i}+F_i)},
\end{align*}
Exploiting $I_2$ in \eqref{eq: I_2}, we can obtain $E_{-i}, E_i$ for $i=1,...,n$ given in \eqref{eq: E_-i} \eqref{eq: E_i}.

In this way, the closed form of the MMSE of estimating $X_1$ from $Y_1$ becomes
\begin{align*}
    &\mathbb E[(X_1 - (\mathbb E[X_1|Y_1]))^2]\\
    &=  \mathbb{E}[X_1^2] - \mathbb{E}\sbrackets{\left(\mathbb{E}\sbrackets{X_1|Y_1}\right)^2}\\
    \!&\!=\! \sum_{i=1}^n\!\brackets{\!\int_{-B_{i+1}}^{-B_i}\!\!\!(x+a_i)^2\!f_{X_0}(x)dx\! + \!\int_{B_i}^{B_{i+1}}\!\!\!(x-a_i)^2\!f_{X_0}(x)dx\!}\\
    &\quad - \int f_{Y_1}(y)\mathbb E^2[X_1|Y_1=y]dy\\
    &=Q - 4\sqrt{Q}\sum_{i=1}^na_i\brackets{\phi\brackets{\frac{B_i}{\sqrt{Q}}} - \phi\brackets{\frac{B_{i+1}}{\sqrt{Q}}}}\\
    &\quad +2\sum_{i=1}^na_i^2p_i - \int \frac{\brackets{\sum_{i=1}^n (E_{-i}+E_i)}^2}{\sum_{i=1}^n (F_{-i}+F_i)} dy.
\end{align*}
This completes our proof of the estimation cost $S(\mathbf{a}_1^n, \mathbf{B}_1^n)$ in Theorem \ref{thm: lope}.
\end{proof}
% \begin{table}
% \caption{An Example of a Table}
% \label{table_example}
% \begin{center}
% \begin{tabular}{|c||c|}
% \hline
% One & Two\\
% \hline
% Three & Four\\
% \hline
% \end{tabular}
% \end{center}
% \end{table}

%%%%%%%%%%%%%%%%%%%%%%%%%%%%%%%%%%%%%%%%%%%%%%%%%%%%%%%%%%%%%%%%%%%%%%%%%%%%%%%%
\vspace{0.2cm}
\section*{Appendix B: Proof of Theorem \ref{thm: first order opt}}
\begin{proof}
To show the first part of the theorem, we recall the optimal linear strategy given in the following lemma:
\begin{lemma}[\!\!\protect{\cite[Lemma 11]{witsenhausen1968}}]
    The best linear policy is $U_1=-\sqrt{\frac{P}{Q}}X_0$, if $P\leq Q$, otherwise $U_1=-X_0 + \sqrt{P-Q}$, which induces the estimation cost
    \begin{align}
        S_{\ell}(P)=\begin{cases}\frac{(\sqrt{Q}-\sqrt{P})^{2}\cdot N}{(\sqrt{Q}-\sqrt{P})^{2}+N} & \text{ if } P\in[0,\ Q],\\ 0 &\text{ otherwise } .\end{cases}\label{eq: opt linear cost}
    \end{align}
\end{lemma}
Given the cost function \eqref{eq: opt linear cost}, we can easily obtain that
\begin{align*}
    \lim_{P\rightarrow 0^+}\frac{d S_{\ell}(P)}{d P} = \lim_{P\rightarrow 0^+}-\frac{N^2(\sqrt{Q}-\sqrt{P})}{\sqrt{P}[(\sqrt{Q}-\sqrt{P})^2+N]^2}= -\infty.
\end{align*}

Next, to demonstrate first-order optimality for the LoPE scheme with costs \eqref{eq: power cost n step} \eqref{eq: est cost n step} in the low-power regime, it suffices to analyze the scheme with the fewest tuning parameters. This is because if a strategy with the minimum number of parameters achieves the optimal rate of decay, then the more expressive strategies with more parameters will also satisfy the optimality condition. 
    
Hence, let's consider the simplest case of the LoPE strategy, the BPSK strategy given by \eqref{eq: BPSK}. 
% Note that BPSK differs from a general 2-step LoPE strategy (take $n=2$ in \eqref{eq: U_1}), since here we also specify $B_2=0$, making the strategy even simpler.
The power and estimation costs in this case become
    \begin{align*}
        &P_{\text{BPSK}}(a) = a^2,\\
    & S_{\text{BPSK}}(a) = \underbrace{Q-2a\sqrt{\frac{2Q}{\pi}}+ 
    a^2}_{T_1(a)}  - \underbrace{\int\frac{(E_1+E_{-1})^2}{F_1+F_{-1}}dy}_{T_2(a)},
    \end{align*}
    where term $T_1(a)$ is the second moment of $X_1$ and $T_2(a)$ is the power of the MMSE estimator for $X_1$ given $Y_1$, i.e., $\mathbb{E}[\left(\mathbb{E}\sbrackets{X_1|Y_1}\right)^2]$.
    Using the chain rule of the derivative, we can obtain
    \begin{align*}
         \frac{\partial S_{\text{BPSK}}(a)}{\partial P_{\text{BPSK}}(a)}  =\frac{\partial S_{\text{BPSK}}(a)}{\partial a}\frac{\partial a}{\partial a^2}=\frac{1}{2a}\brackets{\frac{\partial T_1(a)}{\partial a} + \frac{\partial T_2(a)}{\partial a}}.
    \end{align*}
Hence, if $\frac{\partial T_1(a)}{\partial a} + \frac{\partial T_2(a)}{\partial a}$ remains a negative finite value when $a\rightarrow 0^+$, the limit of the derivative $\frac{\partial S_{\text{BPSK}}(a)}{\partial P_{\text{BPSK}}(a)}\rightarrow -\infty$.

The first derivative term can be easily derived at $a=0$:
\begin{align*}
    \frac{\partial T_1(a)}{\partial a}\bigg|_{a=0} = -2\sqrt{\frac{2Q}{\pi}}.
\end{align*}

On the other hand, a standard dominated–convergence argument shows that, provided the integrand does not blow up pathologically as $a\rightarrow 0$, we can pass the derivative inside. Therefore, given that $T_2(a)$ is finite and $E_i,F_i$ is smooth w.r.t. $a$ for all $y$, we obtain that
\begin{align*}
    \frac{\partial T_2(a)}{\partial a}\bigg|_{a=0} = \int  \frac{\partial}{\partial a}\frac{(E_1(a,y)+E_{-1}(a,y))^2}{F_1(a,y)+F_{-1}(a,y)} \bigg|_{a=0} dy.
\end{align*}
Moreover, we have
\begin{align}
     &\frac{\partial}{\partial a}\frac{(E_1(a,y)+E_{-1}(a,y))^2}{F_1(a,y)+F_{-1}(a,y)} \bigg|_{a=0}\nonumber\\
     &= \frac{2(E_1(0,y)+E_{-1}(0,y))(\frac{\partial E_1(a,y)}{\partial a}\big|_{a=0}+\frac{\partial E_{-1}(a,y)}{\partial a}\big|_{a=0})}{(F_1(0,y)+F_{-1}(0,y))}\nonumber\\
     &\quad - \frac{(E_1(0,y)+E_{-1}(0,y))^2(\frac{\partial F_1(a,y)}{\partial a}\big|_{a=0}+\frac{\partial F_{-1}(a,y)}{\partial a}\big|_{a=0})}{(F_1(0,y)+F_{-1}(0,y))^2}\nonumber\\
     &=2\frac{yQ}{Q+N}\brackets{\frac{\partial E_1(a,y)}{\partial a}\bigg|_{a=0}+\frac{\partial E_{-1}(a,y)}{\partial a}\bigg|_{a=0}}\label{eq: sum of E_i,F_i derivative}\\
     &\quad - \brackets{\frac{ yQ}{Q+N}}^2\brackets{\frac{\partial F_1(a,y)}{\partial a}\bigg|_{a=0}+\frac{\partial F_{-1}(a,y)}{\partial a}\bigg|_{a=0}}.\nonumber
\end{align}

Lemma \ref{lemma: gaussian derivation} in the follows states that if a function is a product of Gaussian PDFs $\phi(\cdot)$ and CDFs $\Phi(\cdot)$, then, under derivation, it remains a product of $\phi(\cdot)$ and $\Phi(\cdot)$, multiplied by some polynomial factors.

\vspace{-0.1cm}
\begin{lemma}\label{lemma: gaussian derivation}
Let $g_1,...,g_n$ be polynomials of $x$. For some nonnegative parameters $\alpha_1,...,\alpha_n,\beta_1,...,\beta_n$, we define
\begin{align*}
    f(x) = \prod_{i=1}^n\brackets{\phi\brackets{g_i(x)}}^{\alpha_i}\brackets{\Phi\brackets{g_i(x)}}^{\beta_i}.
\end{align*}
Then, the derivative of $f(x)$ can be written as
\begin{align*}
       \frac{df(x)}{dx} =\sum_{k=1}^n \sbrackets{ p_k(x)\prod_{i=1}^n\brackets{\phi\brackets{g_i(x)}}^{\alpha_{i,k}'}\brackets{\Phi\brackets{g_i(x)}}^{\beta_{i,k}'}},
 \end{align*}
    where each $p_k(x)$ is a polynomial of $x$ and $\alpha_{i,k}',...,\beta_{i,k}'$ are some nonnegative parameters for $i,k=1,...,n$.
\end{lemma}
\vspace{-0.1cm}
The proof of this result follows directly from the derivatives of Gaussian PDF and CDF:
\begin{align*}
    \frac{d}{dx}\phi(x) = -x\phi(x),\text{ }\frac{d}{dx}\Phi(x) = \phi(x).
\end{align*}

Now, since $F_i(a,y), E_i(a,y),i=\pm 1$ are composed entirely of sums and products of Gaussian PDFs and CDFs, then by Lemma \ref{lemma: gaussian derivation}, their derivatives appeared in term \eqref{eq: sum of E_i,F_i derivative} are also composed of Gaussian PDFs and CDFs, multiplied by some polynomials. Given the exponential tail decay of Gaussian distributions which surpasses the decay rate of all the polynomials, these derivative terms are integrable, and hence the overall integral converges:
\begin{align*}
   \frac{\partial T_2(a)}{\partial a}\bigg|_{a=0} = \int  \frac{\partial}{\partial a}\frac{(E_1(a,y)+E_{-1}(a,y))^2}{F_1(a,y)+F_{-1}(a,y)} \bigg|_{a=0} dy \triangleq C,
\end{align*}
where $C$ is some finite value. 

Next, to deduce the sign of the constant $C$, note that taking the positive $ a\rightarrow 0^+$ implies $X_1\rightarrow X_0$. In this case, the system state to be estimated is made less concentrated and therefore has more variance. Hence, the power of its MMSE estimator $T_2(a) = \mathbb{E}[\left(\mathbb{E}\sbrackets{X_1|Y_1}\right)^2]$ increases accordingly, which indicates its derivative $\frac{\partial T_2(a)}{\partial a}\big |_{a=0} = C\geq 0$.

Overall, we have shown that 
\begin{align*}
    \lim_{a\rightarrow 0}\frac{\partial S_{\text{BPSK}}(a)}{\partial P_{\text{BPSK}}(a)}  
    =\lim_{a\rightarrow 0}\sbrackets{-\frac{1}{a}\sqrt{\frac{2Q}{\pi}} - \frac{C}{2a}}
    =-\infty.
\end{align*}
    
\end{proof}

\bibliographystyle{ieeetr}
\bibliography{IEEEabrv,main}

\end{document}